\documentclass[A4, 11pt]{article}
\usepackage{latexsym}
\usepackage{amsmath}
\usepackage{amsthm}
\usepackage{amssymb}
\usepackage{amsfonts}
\usepackage{esint}
\usepackage[all]{xy}
\usepackage{graphics,graphicx}
\usepackage{accents}
\usepackage[english]{babel}
\usepackage[utf8]{inputenc}
\usepackage[T1]{fontenc}
\usepackage{color}
\usepackage{tikz-cd}
\usepackage{hyperref}
\usepackage{enumitem}
\theoremstyle{definition}
\newtheorem{defin}{Definition}[section]
\theoremstyle{plain}
\newtheorem{theo}[defin]{Theorem}
\newtheorem{prop}[defin]{Proposition}
\newtheorem{cor}[defin]{Corollary}
\newtheorem{lem}[defin]{Lemma}
\theoremstyle{definition}
\newtheorem{rem}[defin]{Remark}
\newtheorem{ex}[defin]{Example}

\def\A{{\mathbf A}}

\def\N{{\mathbf N}} 
 
\def\Q{{\mathbf Q}}
\def\Z{{\mathbf Z}}

\def\F{{\mathbf F}}

\def \mA { \mathcal{A}}

\def\Ker{ \operatorname{Ker} }
\def\Im{ \operatorname{Im} }
\def \Spec { \operatorname{Spec} }

\def \Hom {\operatorname{Hom} }

\def \Card {\operatorname{Card} }
\def \Gal {\operatorname{Gal} }
\def \lcm {\operatorname{lcm} }
\def \Aut {\operatorname{Aut} }
\author{Séverin Philip}
\begin{document}
\title{On the semi-stability degree for abelian varieties}
\maketitle

\begin{abstract}
For an abelian variety $A$ over a number field we study bounds depending only on the dimension of $A$ for the minimal degree $d(A)$ of a field extension over which $A$ acquires semi-stable reduction. We first compute $d(A)$ in terms of the cardinalities of the finite monodromy groups of $A$ which leads to a bound on $d(A)$ in terms of the classical Minkowski bound. We then show this bound is tight up to its $2$-part by considering $p$-adic coverings of the local points of a universal abelian scheme.
\end{abstract}

\section{Introduction}

In \cite{sga} exposé IX Grothendieck proves several seminal results related to the semi-stable reduction of abelian varieties among which the semi-stable reduction theorem. This theorem allows us to make the following definition.
\begin{defin} Let $g$ be a positive integer.
For an abelian variety $A$ over a number field $K,$ we set 
$$d(A)=\min\{ [L:K] \mid L/K \text{ finite,}~ A_L \text{ has semi-stable reduction}\}.$$
For a number field $K,$ we set
$$d_g(K)=\underset{B/K \text{ p.p. } \dim B=g}{\lcm} d(B)$$
 where p.p. stands for principally polarised.
And finally we set 
$$d_g=\underset{\dim B=g}{\sup}d(B),$$
the supremum being taken over all abelian varieties of dimension $g$ over some number field.

\end{defin}

We prove that the lowest common multiple $d_g(K)$ for any number field $K$ is realized as the value $d(A)$ of some abelian variety $A$ over a number field $L/K$.

\begin{theo}\label{maintheo}
For any number field $K$ and nonzero integer $g$ there is a finite extension $L/K$ with a principally polarised abelian variety $A$ of dimension $g$ over $L$ such that
$$d(A)=d_g(K).$$
\end{theo}

In a previous work \cite{Ph2}, we built for every positive integer $g$ (using wild finite monodromy, see below) and for each odd prime $p$ a (CM) abelian variety $B_p$ of dimension $g$ over a number field $K_g$ such that the $p$-adic valuation of $d(B_p)$ is given by 
$$r(g,p)=\sum\limits_{i=0}^{\infty} \lfloor \frac{2g}{p^i(p-1)}\rfloor.$$
The numbers $r(g,p)$ together form the Minkowski bound
$$M(2g)=\prod\limits_p p^{r(g,p)}$$
which is also given by the lowest common multiple of the cardinalities of the finite subgroups of $\mathrm{GL}_{2g}(\Q)$. 

This leads to the following corollary.

\begin{cor}\label{maincor}
Let $g$ be a positive integer. Then there is an abelian variety $A$ of dimension $g$ over some number field such that 
$$\frac{M(2g)}{2^{g-1}} = d(A).$$
Furthermore we have the inequalities
$$\frac{M(2g)}{2^{g-1}}\leq d_g \leq M(2g).$$
\end{cor}

The main interest of the corollary is to give an almost sharp estimate for $d_g$. Note also that the upper bound given here is a strong improvement from the most commonly used one $d_g\leq \Card \mathrm{GL}_{2g}(\Z/12\Z)$ which can be deduced from Proposition 4.7 of \cite{sga} exposé IX. Indeed, this can already be seen on the first values 
$$\begin{array}{c|c|c|c|c}
g & 1 & 2 & 3 & 4   \\
M(2g) & 24 & 5760 & 2903040 & 1393459200  \\
\operatorname{Card} \mathrm{GL}_{2g} (\Z/12\Z) & 4608 & \simeq 3.2\times 10^{16} & \simeq 1.2\times 10^{38} & \simeq 1.9\times 10^{68} \\
\end{array}$$
but also from the asymptotics $\lim\limits_{n\rightarrow \infty} (\frac{M(n)}{n!})^{1/n}\simeq 3.4109$ (see \cite{Ka}) and $$\lim\limits_{n\rightarrow \infty}(\operatorname{Card} \mathrm{GL}_{n} (\Z/12\Z))^{1/n^2}=12.$$ 

We now briefly expose the content of this paper. We first relate $d(A)$ to the finite monodromy groups of $A$, where $A$ is an abelian variety over a number field $K$. These groups, which represent the local obstruction to semi-stable reduction, were first introduced by Serre in the case of elliptic curves in \cite{propgal} and generalized by Grothendieck to any dimension in \cite{sga}. For a non-archimedean place $v$ of $K$ we denote by $\Phi_{A,v}$ the finite monodromy group of $A$ at $v$. Let $\Sigma_K$ be the set of non-archimedean places of $K$, then we get the following formula
\begin{equation} \label{lcm=d(A)} d(A)=\underset{v\in \Sigma_K}{\lcm}\Card \Phi_{A,v}.
\end{equation}
This with the divisbility bound $\Card \Phi_{A,v} \mid M(2g)$ proved by Silverberg and Zarhin in \cite{Sb1998} yields the upper bound of Corollary \ref{maincor} through the divisibility relations
$$d(A)\mid M(2g) \text{ and } d_g\mid M(2g).$$

The relation between finite monodromy groups and $d(A)$ also opens a way to build abelian varieties with maximal $d(A)$ knowing only local data which is done by Theorem \ref{maintheo}. In order to achieve this we look at the geometric behavior of finite monodromy groups in famillies of abelian varieties. More precisely we study their variation in the fibers of abelian schemes $\mA \rightarrow S$. We are able to replace the abelian scheme $\mA$ by its $\ell$-torsion subscheme which is a finite étale cover of $S.$ This comes from the fact that the finite monodromy groups of $A$ can be read on the Galois action on the $\ell$-torsion given $\ell$ is a prime big enough (see Proposition \ref{monodromyltors}). For such a cover we show that the image of Galois in its fibers is locally constant for the $v$-adic topologies and so are the finite monodromy groups in the fibers of abelian schemes. To recover a global object from this situation we thus require some results on weak approximation on the base scheme.

We wish to apply these results to a universal abelian scheme for abelian varieties of a given dimension $g$. A substitute to such a scheme is constructed in the chapter 7 of \cite{mumgit}. The base $H_g$ of this abelian scheme provided by Mumford is a moduli space for principally polarised abelian varieties of dimension $g$ with some rigidification. It is not known whether $H_g$ satisfies weak approximation or not. To remedy this issue we use a result of Ekedahl on Hilbert's irreducibility theorem for finite étale covers which allows us to solve our weak approximation problem up to some uncontrolled finite field extension. This last part explains the presence of the field $L$ in Theorem \ref{maintheo}. As a consequence we also obtain a form of local-global principle for finite monodromy groups.

\section{Finite monodromy groups and $d(A)$}

\subsection{Local situation}

In this section $A$ is an abelian variety over a local field $K_v$ with valuation $v$ of residue characteristic $p$. Let us denote by $K_v^{\mathrm{un}}$ the maximal unramified extension of $K_v$. It follows from \cite{sga} exposé IX section 4.1 and Grothendieck's Galois criterion for semi-stability that for an abelian variety $A$ over $K_v$ there is a smallest extension $(K_v^{\mathrm{un}})_{A,s}$ of $K_v^{\mathrm{un}}$ over which $A$ acquires semi-stable reduction. Its Galois group is $\Phi_{A,v}$ by definition. In particular if $L_v$ is an extension of $K_v$ then $A_{L_v}$ has semi-stable reduction if and only if $(K_v^{\mathrm{un}})_{A,s}\subset L_vK_v^{\mathrm{un}}$. 
Here we first show that we can descend the extension $(K_v^{\mathrm{un}})_{A,s}/K_v^{\mathrm{un}}$ to an extension of the same degree over $K_v$ over which $A$ acquires semi-stable reduction. This is done at the cost of the Galois property of $(K_v^{\mathrm{un}})_{A,s}/K_v^{\mathrm{un}}$.

\begin{lem}\label{carfinimonogrp}
There is an extension $L_v/K_v$ of degree $\Card \Phi_{A,v}$ such that $A_{L_v}$ has semi-stable reduction. Moreover, if $L_v/K_v$ is a finite extension such that $A_{L_v}$ has semi-stable reduction then $\Card \Phi_{A,v}$ divides the ramification index of $L_v/K_v$.
\end{lem}
\begin{proof}

By section 4.1 of \cite{sga} exposé IX the extension $(K_v^{\mathrm{un}})_{A,s}$ is Galois over $K_v$ and we have an exact sequence
$$\begin{tikzcd}
1 \arrow[r] & \Phi_{A,v} \arrow[r] & \Gal ((K_v^{\mathrm{un}})_{A,s}/K_v) \arrow[r] & \widehat{\Z} \arrow[r] & 1.
\end{tikzcd}$$
As $\widehat{\Z}$ is projective (see proposition 5.2.2 of \cite{Wils}) this sequence admits a splitting $s$. The subgroup $H=s(\widehat{\Z})$ is closed as the continuous image of a compact, verifies $H\cap \Phi_{A,v}=\{1\}$ and is of index $\Card \Phi_{A,v}$. It is therefore an open subgroup and corresponds to an extension $L_v/K_v$ of degree $\Card \Phi_{A,v}$. Moreover we have $L_vK_v^{\mathrm{un}}=(K_v^{\mathrm{un}})_{A,s}$ by construction so that $A_{L_v}$ has semi-stable reduction.

For the second part of the theorem, let $L_v/K_v$ be such that $A_{L_v}$ has semi-stable reduction. Then we have $I_{L_v}\subset I_{A,v}$ so that the ramification index $e(L_v/K_v)=[I_{K_v}:I_{L_v}]$ is divisible by $\Card \Phi_{A,v}$. 
\end{proof}

The following proposition will enable us to recover the finite monodromy groups of the fibers of an abelian scheme in the next section.
\begin{prop} \label{monodromyltors}Let $\ell>\max(2\dim A+1,p)$ be a prime number. The group $\Gal (K_v^{\mathrm{un}}(A[\ell])/K_v^{\mathrm{un}})$ is either $\Phi_{A,v}$ or $\Phi_{A,v}\times \Z/\ell \Z$.
\end{prop}

\begin{proof}
By Proposition 4.7 of \cite{sga} exposé IX we have $(K_v^{\mathrm{un}})_{A,s}\subset K_v^{\mathrm{un}}(A[\ell])$ and by Proposition 3.5 of \textit{loc. cit.} an element $\sigma\in \Gal(\overline{K_v}/(K_v^{\mathrm{un}})_{A,s})$ acts unipotently with order $2$ on $A[\ell]$. That is, if $x\in A[\ell]$ we have $(\sigma-\mathrm{id})^2(x)=0$ and so $\sigma^2(x)=2\sigma(x)-x$. We get $\sigma^{\ell}(x)=\ell\sigma(x)-(\ell-1)x=x$ as $x$ is of $\ell$-torsion and that $\Gal(K_v^{\mathrm{un}}(A[\ell])/(K_v^{\mathrm{un}})_{A,s})$ is of exponent $\ell$. Hence this group is either trivial or $\Z/\ell\Z$. 

In the first case $\Gal(K_v^{\mathrm{un}}(A[\ell])/K_v^{\mathrm{un}})=\Phi_{A,v}$ and we are done. In the second case, the prime divisors of $\Phi_{A,v}$ are smaller or equal to $2\dim A+1$ by the divisibility bound $\Card \Phi_{A,v} \mid M(2g)$ so that by choice of $\ell$ the extension $(K_v^{\mathrm{un}})_{A,s}$ is linearly disjoint of the unique extension $K_{v,\ell}^{\mathrm{un}}$ of $K_v^{\mathrm{un}}$ of order $\ell$. As $K_v^{\mathrm{un}}(A[\ell])$ is the unique extension of order $\ell$ of $(K_v^{\mathrm{un}})_{A,s}$ we get that it is the compositum $(K_v^{\mathrm{un}})_{A,s}K_{v,\ell}^{\mathrm{un}}$ and so the statement on its Galois group follows.
\end{proof}

\subsection{Global situation}

The main theorem of this section will follow from the study of the local situation and the following result of weak approximation for fields.
\begin{prop}\label{weakapproxfield}
Let $K$ be a number field and $v_1,\dots, v_n \in \Sigma_K$. Let $L^{(v_i)}/K_{v_i}$ be finite extensions of the same degree $d$. Then there is an extension $L/K$ of degree $d$ such that there is a unique place $w_i$ over  $v_i$ for each $i\in \{1,\dots, n\}$ and it verifies $L_{w_i}\simeq L^{(v_i)}$. 
\end{prop}

This result is proved similarly to other classical results of the same flavor, see for example chapter 6 of \cite{rib} and specially Theorem 4.

\begin{theo} \label{monoglobdeg}
Let $A$ be an abelian variety over a number field $K$. Then there is an extension $L/K$ of degree $\underset{v\in \Sigma_K}{\lcm} \Card \Phi_{A,v}$ such that $A_L$ has semi-stable reduction. Moreover if $L/K$ is a finite extension such that $A_L$ has semi-stable reduction then
$$ \underset{v\in \Sigma_K}{\lcm} \Card \Phi_{A,v} \mid [L:K].$$
\end{theo}
\begin{proof}
Let $d=\underset{v\in \Sigma_K}{\lcm} \Card \Phi_{A,v}$. For $v\in \Sigma_K$ such that $A$ has not semi-stable reduction at $v$ let $d_v$ be such that $d_v\cdot \Card \Phi_{A,v}=d$. By Lemma \ref{carfinimonogrp} there is an extension $L_v$ of $K_v$ of degree $\Card \Phi_{A,v}$ such that $A_{L_v}$ has semi-stable reduction. Let $M_v$ be the unramified extension of $L_v$ of degree $d_v$. Then $M_v/K_v$ is of degree $d$ and we can apply Proposition \ref{weakapproxfield} to the local extensions $M_v/K_v$ for the places $v\in \Sigma_K$ with $\Phi_{A,v}$ non trivial. We get an extension $L/K$ such that $A_L$ has semi-stable reduction by construction.

Now let $L/K$ be an extension such that $A_L$ has semi-stable reduction. Then for every $v\in \Sigma_K$ and every place $w\mid v$ of $L$ by Lemma \ref{carfinimonogrp} we have
$$\Card \Phi_{A,v} \mid [L_w:K_v]$$
so that $d\mid [L:K]$ as the global degree is the sum of the local degrees. 
\end{proof}
\begin{rem}As a consequence of the theorem we get the equality (\ref{lcm=d(A)}) of the introduction $d(A)=\underset{v\in \Sigma_K}{\lcm} \Card \Phi_{A,v}$. 
\end{rem}

With the work of Silverberg and Zarhin we get the divisibility bound $M(2g)$ as a corollary.
\begin{cor}
Let $A$ be an abelian variety of dimension $g$ over a number field. We have the divisibility relations $d(A)\mid M(2g)$ and $d_g\mid M(2g)$. 
\end{cor}
\begin{proof}
By Corollary 6.3 of \cite{Sb1998} we have the divisibility relation $\Card \Phi_{A,v} \mid M(2g)$ for any place $v\in \Sigma_K$. The corollary then follows directly from the equality 
$$d(A)=\underset{v\in \Sigma_K}{\lcm} \Card \Phi_{A,v}$$
given by Theorem \ref{monoglobdeg}.
\end{proof}

\section{Some $p$-adic open coverings of abelian schemes} \label{sectpadiccovering}

\subsection{From finite étale covers of schemes to $p$-adic coverings}

In this section $K$ is a field. The goal here is to give a refined version of Krasner's lemma for finite étale covers as in \cite{po} Proposition 3.5.74 for a local field $L/K$. We first recall the construction and properties of the Galois closure of a finite étale cover of a connected scheme. 

\begin{prop}
Let $S'\rightarrow S$ be a finite étale cover of $K$-schemes with $S$ connected. There is a Galois cover $T\rightarrow S$ such that for every point $s\in S(L)$ with $K\subset L$ and $\overline{s}\in S(\overline{L})$ a geometric point over $s$ the following properties are satisfied.
\begin{itemize}
\item[$(i)$] The set $\Hom_S(T,S')$ is identified with the geometric fiber $S'\times_S \overline{s}$ by any choice of a point in that fiber.
\item[$(ii)$] The action of $\Aut_S T$ on $S'\times_S \overline{s}$ by a choice as in $(i)$ is faithful. 
\item[$(iii)$] The action of $G_L=\mathrm{Gal}(\overline{L}/L)$ on $S'\times_S \overline{s}$ factors through $\Aut_S T$ by any choice as in $(i)$, i.e. it is given by a map
$$\varphi_s \colon G_L \longrightarrow \Aut_S T$$
and different choices of points in the fiber induces conjugated maps.
\end{itemize}
\end{prop}

\begin{proof}
Let $s\in S(L)$ and $\overline{s}\in S(\overline{L})$ a geometric point over $s$. Let us denote by $a_1,\dots, a_n$ the elements of the set $S'\times_S \overline{s}$ and let $Y=S'\times_S \dots \times_S S'$ the product being taken $n$ times. The geometric point $\overline{a}=(a_1,\dots, a_n)\in Y\times_S \overline{s}$ lies over a point $a\in Y$. Let $C$ be the connected component of $Y$ containing $a$. Let $p_1,\dots, p_n$ be the projections $Y\rightarrow S'$ and $\iota\colon C\hookrightarrow Y$ the inclusion. These maps define $n$ elements $p_1\circ \iota, \dots , p_n\circ \iota \in \Hom_S(C,S')$. Furthermore for $i\in \{1,\dots, n\}$ we have $p_i \circ \iota (\overline{a})=a_i$ so that the evaluation map
$$\begin{array}{cccc}
\operatorname{ev}_{\overline{a}}\colon & \Hom_S(C,S') & \longrightarrow & S'\times_S \overline{s} \\
& f & \longmapsto & f(\overline{a})
\end{array}
$$
is surjective. As two finite étale covers from a connected scheme which coïncide on a geometric point are equal $\operatorname{ev}_{\overline{a}}$ is also injective and thus a bijection.

We now prove that $C$ is a Galois cover of $S$. First as $C$ is connected the action of $\Aut_S C$ on $C\times_S \overline{s}$ is free. We show it is transitive. Let $\overline{a'}\in C\times_S \overline{s}$.  As before the evaluation map
$$\begin{array}{cccc}
\operatorname{ev}_{\overline{a'}}\colon & \Hom_S(C,S') & \longrightarrow & S'\times_S \overline{s} \\
& f & \longmapsto & f(\overline{a'})
\end{array}
$$
is injective hence bijective as the sets have the same cardinality. We also have that $\Hom_S(C,S')=\{p_1\circ \iota, \dots, p_n\circ \iota\}$ so we get $\overline{a'}=(a_{\sigma(1)},\dots, a_{\sigma(n)})$ for some permutation $\sigma\in \mathfrak{S}_n$. Now $\sigma$ defines an automorphism of $Y$ by permuting the factors and $\overline{a'}\in \sigma(C)$ so that $\sigma(C)=C$. The permutation $\sigma$ thus defines an automorphism of $C$ which sends $\overline{a}$ to $\overline{a'}$ and so the action is transitive which proves that $C$ is indeed Galois over $S$. 

We now show that the action of $\Aut_S C$ on $\Hom_S(C,S')$ is faithful which gives $(ii)$. Let $g\in \Aut_S C$ be such that $g$ acts trivially on $\Hom_S(C,S')$. It follows that for all $i\in \{1,\dots,n \}$ we have 
$$p_i\circ \iota (g ( \overline{a}))=a_i$$
so $g(\overline{a})=\overline{a}$ and $g=\mathrm{id}_C$. 

We check property $(iii).$ The action of $G_L$ on $C\times_S \overline{s}$ commutes to that of $\Aut_S C$ so that by the choice of a point in the geometric fiber it is given by a map $\varphi_s\colon G_L\rightarrow \Aut_S C$ as the latter action is free and transitive. Now as the following diagram commutes
$$\begin{tikzcd}
C\times_S \overline{s} \arrow[r, "\sigma"] \arrow[d, "f"]& C\times_S \overline{s} \arrow[d, "f"]\\
S'\times_S \overline{s} \arrow[r, "\sigma"] & S'\times_S \overline{s}.
\end{tikzcd}$$
the Galois action on $S'\times_S \overline{s}$ is deduced from the action on $C\times_S \overline{s}$ which yields $(iii)$.

It remains to see that this construction is independent of the point $s\in S(L)$. Let $t\in S(L)$ and $\overline{t}\in S(\overline{L})$ be another choice of $L$-point and geometric point. The same construction yields a Galois cover $T$ of $S$ with the same properties with regards to $\overline{t}$. As $S$ is connected the cardinality of the geometric fibers of $S'\rightarrow S$ is constant and every connected component of $S'$ surjects onto $S$. Let $S'\times_S \overline{t}=\{b_1,\dots, b_n\}$ and assume without loss of generality that $b_1$ and $a_1$ lie over points in the same connected component of $S'$. Let $\pi_1\in \Hom_S(T,S')$ be such that $\pi_1(\overline{b})=b_1$ for some $\overline{b}\in T(\overline{L})$ and denote $\pi_2,\dots, \pi_n$ the other elements of $\Hom_S(T,S')$. The map $\pi_1$ induces a surjection from $T$ on the connected component of $S'$ containing $b_1$ so we get a geometric point $\overline{c}\in T(\overline{L})$ such that $\pi_1(\overline{c})=a_1$. Now the map
$$\begin{array}{cccc}
\operatorname{ev}_{\overline{c}}\colon & \Hom_S(T,S') & \longrightarrow & S'\times_S \overline{s} \\
& f & \longmapsto & f(\overline{c})
\end{array}
$$
is injective between sets of the same cardinality so it is bijective. Up to renumbering the maps $\pi_1,\dots, \pi_n$ define a map $\eta\colon T\rightarrow Y$ such that $p_i\circ \eta= \pi_i$ and
$$\pi_i\circ \eta(\overline{c})=a_i.$$
We get that $\eta(\overline{c})=\overline{a}$ and $T$ maps to $C$. By the same argument we get a map $C\rightarrow T$ and thus $T\simeq C$. 
\end{proof}

From now on, for a finite étale cover of schemes $S'\rightarrow S$ let $T$ be a Galois cover satisfying the properties of the proposition and $\varphi_s$ be the map given by $(iii)$ for an $L$-point $s$ of $S$. The following lemmas will enable us to check that the open subsets coming from the maps $\pi_H\colon T_H\rightarrow S$ for $H\subset G$ have fibers with constant Galois groups where we denote by $T_H$ the subcover of $T$ associated to $H$.

\begin{lem}\label{galgrofdef}
Let $S'\rightarrow S$ be a finite étale cover of $K$-schemes and $L/K$ be an extension of fields. For every $s\in S(L)$ there is an isomorphism
$$G_L/\Ker \varphi_s \simeq \Gal (L(S'_s(\overline{L}))/L).$$
\end{lem}

\begin{proof}
By definition $\Gal(L(S'_s(\overline{L}))/L)$ is the quotient $G_L/\Ker \rho_s$ where 
$$\rho_s \colon G_L\rightarrow \Aut S'\times_S \overline{s}$$
 is the natural Galois action. By property $(iii)$ of $T$ this action factor through $\varphi_s \colon G_L\longrightarrow \Aut_S T$. As $\Aut_S T$ acts faithfully on $S'\times_S \overline{s}=\Hom_S(T,S')$ we have $\Ker \varphi_s=\Ker \rho_s$. 
\end{proof}

\begin{lem}\label{ratpointconjug}
Let $S'\rightarrow S$ be a finite étale cover of $K$-schemes and $L/K$ be an extension of fields. Let $s\in S(L)$ and $H\subset \Aut_S T$ be a subgroup. The fiber $T_H\times_S s$ has an $L$-point if and only if there is a $g\in \Aut_S T$ such that $\Im \varphi_s \subset gHg^{-1}$. 
\end{lem}
\begin{proof}
As before we identify $\Hom_S(T,T_H)$ and $T_H\times_S \overline{s}.$ The action of $G_L$ on $T_H\times_S \overline{s}$ factor through $\operatorname{Aut}_S T$ by $\varphi_s$. The stabilizer of the canonical map $p_H\colon T\rightarrow T_H$ by the action of $\Aut_S T$ is $H$ so, as the action is transitive, the stabilizer of any element in $T_H\times_S \overline{s}$ is a conjugate of $H$.   We can now prove the equivalence.

An $L$-point of $T_H\times_S s$ corresponds to an element of $\Hom_S(T,T_H)$ fixed by $G_L$. It follows that $\Im \varphi_s$ is a subgroup of its stabilizer, that is a conjugate of $H$. Conversely, if $\Im \varphi_s$ lie in some conjugate $gHg^{-1}$ of $H$ it fixes $g\cdot p_H\in \Hom_S(T,T_H)$ which yields an $L$-point of the fiber.
\end{proof}

We can now prove the main result of this subsection.

\begin{theo}\label{etanalcover}
Let $S'\rightarrow S$ be a finite étale cover of $K$-schemes of finite type. Let $L/K$ be a local field. Then there are finite groups $H_1,\dots, H_n$ and a finite covering of $S(L)$ by disjoint open sets $(U_i)_{i\in\{1,\dots, n\}}$ such that, for all $i\in \{1,\dots, n\}$ and $s\in S(L)$,
$$s\in U_i\iff \Gal (L(S'_s(\overline{L}))/L)\simeq H_i.$$ 
\end{theo}
\begin{proof}
For $H\subset \Aut_S T$ let $\pi_H\colon T_H\rightarrow S$ be the canonical map. Consider the subsets 
$$\widetilde{U}_H=\pi_H(T_H(L))\setminus \bigcup\limits_{G\subsetneq H} \pi_G (T_G(L))$$
of $S(L)$. As the maps $\pi_G$ for $G$ a subgroup of $\Aut_S T$ are finite étale they induce open and closed maps for the analytic topologies on $T_G(L)$ and $S(L)$. Hence the sets $\widetilde{U}_H$ are open. The familly of those open sets for varying $H\subset \Aut_S T$ is a covering by construction. 

Now choose representatives $H_1,\dots, H_n$ of the isomorphism classes of the subgroups of $\Aut_S T$ and put $U_{i}=\bigcup\limits_{\underset{H'\simeq H_i}{H'\subset \Aut_S T}} \widetilde{U}_{H'}$. The sets $(U_{i})_{i\in\{1,\dots,n\}}$ form an open covering of $S(L)$. 

We now prove the equivalence
$$s\in U_{i} \iff \Im \varphi_s \simeq H_i$$
for $s\in S(L)$. By definition of $U_i$, if $s\in U_i$ we have that $s$ is in $\widetilde{U}_H$ for some subgroup $H\subset \Aut_S T$ with $H\simeq H_i$. It follows that the fiber $T_H\times_S s$ has an $L$-point so that by Lemma \ref{ratpointconjug} we have $\Im \varphi_s \subset gHg^{-1}$ for some $g\in \Aut_S T$. Now if we have a containment $\Im \varphi_s\subset gGg^{-1}$ for a subgroup $G\subsetneq H$ then again by Lemma \ref{ratpointconjug} we get that the fiber $T_G\times_S s$ has an $L$-point so that $s\in \pi_G(T_G(L))$ which is impossible by definition of $\widetilde{U}_H$. It follows that $\Im \varphi_s = gHg^{-1} \simeq H_i$. The other direction is proved in a similar way.

We also get from the equivalence that the open sets $(U_i)_{i\in\{1,\dots, n\}}$ are disjoint from each other.

We conclude the proof by Lemma \ref{galgrofdef} which gives
$$\Im\varphi_s\simeq H_i \iff G_L/\Ker \varphi_s \simeq H_i \iff \Gal (L(S'_s(\overline{L}))/L)\simeq H_i.$$

\end{proof}

\subsection{The case of an abelian scheme}

An abelian scheme $A\rightarrow S$ comes with finite étale covers $A[n]\rightarrow S,$ for any integer $n\in \N,$ given by the kernel of the multiplication by $n$ map. Proposition \ref{carfinimonogrp} then allows us to relate the analytic coverings of the base scheme $S$ obtained by Theorem \ref{etanalcover} to the finite monodromy groups of the fibers of $A$. 
\begin{theo}\label{monocover}
Let $A\rightarrow S$ be an abelian scheme with $S$ over a number field $K$ and connected. Let $v\in \Sigma_K$ of residue characteristic $p$. Then there are finite groups $(G_i)_{1\leq i\leq n}$ and a finite covering of $S(K_v)$ by disjoint open sets $(U_i)_{1\leq i\leq n}$ such that for $s\in S(K_v)$ we have
$$s\in  U_i \iff \Phi_{A_s,v}\simeq G_i.$$
In particular if $L/K$ is a finite extension with a place $w\mid v$ such that $L_w=K_v$ then for a point $s\in S(L)$ we have
$$ s_{L_w}\in U_i \iff \Phi_{A_s,w}\simeq G_i.$$
\end{theo}
\begin{proof}
By Theorem \ref{etanalcover} applied to the finite étale cover of the $\ell$-torsion $A[\ell]\rightarrow S$ with $\ell>\max (2\dim A+1, p)$ and the local field $K_v^{\mathrm{un}}$ there are finite groups $H_1,\dots, H_n$ and a covering $(V_i)_{i\in \{1,\dots, n\}}$ of $S(K_v^{\mathrm{un}})$ such that
$$s\in V_i \iff \Gal (K_v^{\mathrm{un}}(A_s[\ell])/K_v^{\mathrm{un}})\simeq H_i.$$

By Proposition \ref{monodromyltors} we have that $\Gal(K_v^{\mathrm{un}}(A_s[\ell])/K_v^{\mathrm{un}})$ is isomorphic either to $\Phi_{A_s,v}$ or $\Phi_{A_s,v}\times \Z/\ell \Z$. For all $i\in \{1,\dots, n\}$ such that $\Z/\ell\Z$ is a direct factor of $H_i$ let $G_i=H_i/(\Z/\ell \Z)$. For the remaining indices let $G_i=H_i$. Now put
$$U_i^{\mathrm{un}}=\bigcup\limits_{G_j\simeq G_i} V_j$$
and renumber the $U_i^{\mathrm{un}}$ to remove the redundancies so that we have
$$s\in U_i^{\mathrm{un}} \iff \Phi_{A_s,v}\simeq G_i.$$

The injective map $S(K_v)\hookrightarrow S(K_v^{\mathrm{un}})$ is continuous so that we get the desired covering by taking $U_i=U_i^{\mathrm{un}} \cap S(K_v)$.
\end{proof}

As a direct consequence for a base scheme $S$ satisfying weak approximation we get fibers with prescribed finite monodromy groups.
\begin{cor} \label{corweakapprox}
Let $A\rightarrow S$ be an abelian scheme wtih $S$ over a number field $K$ and satisfying weak approximation. Let $v_1,\dots,v_n\in \Sigma_K$ and $s_i\in S(K_{v_i})$ be local points of $S$ for each $i\in \{1,\dots, n\}$. Then there is a point $s\in S(K)$ with
$$\Phi_{A_s,v_i}\simeq \Phi_{A_{s_i},v_i}.$$ 
\end{cor}

We end this section by an example of elliptic scheme over the affine line minus one point.

\begin{ex}
Let us consider the abelian scheme of dimension $1$ over $S=\A^1_{\Q}\setminus \{0\}$ given by the equation
$$E\colon y^2=x^3+t.$$

For any rational point $s\in S(\Q)$ the elliptic curve $E_s$ fiber of $E$ at $s$ has equation
$$y^2=x^3+s$$
hence a discriminant $\Delta_{E_s}=-2^4\cdot 3^3 \cdot s^2$ and has $0$ as $j$-invariant. Furthermore we have $c_4=0$ and $c_6=2^5\cdot 3^3\cdot s$ for the classical invariants $c_4$ and $c_6$. We look at the reduction at $2$ and $3$ of the elliptic curve in the familly $E$. The fact that the $j$-invariant is $0$ gives that these curves have potential good reduction. We are going to construct from the work of Kraus in \cite{kraus} explicit open subsets of the open coverings given by Theorem \ref{monocover} for the $2$-adic and $3$-adic topologies on $S$. We will restrict ourselves to small valuations for $s$ to have a Weierstrass equation that is minimal at $2$ and $3$.

We start by looking at the reduction at $3$. The table page 356 of \cite{kraus} gives the finite monodromy group of $E_s$ at $3$ for $s\in S(\Q)$ depending on the $3$-adic valuations of the invariants of $E_s$ and their reduction $\mod 9$. We treat some cases depending on $v_3(s)$ :
\begin{itemize}[leftmargin=*]
\item $v_3(s)=0$:~ $\Phi_{E_s,3}=\Z/4\Z$ if $s\equiv 1\text{ or } 8\mod 9$ and $\Phi_{E_s,3}=\Z/3\Z\rtimes \Z/4\Z$ otherwise. 
\item $1\leq v_3(s)\leq 4$:~ $\Phi_{E_s,3}=\Z/3\Z\rtimes \Z/4\Z$ unless $v_3(s)=3$. In that last case we have $\Phi_{E_s,3}=\Z/4\Z$ if, by noting $s=3^3u$, we have $u\equiv 1\text{ or } 8 \mod 9$. 
\end{itemize}
We see from these computations that the $3$-adic open subset of $S$ corresponding to the group $\Z/4\Z$ by Theorem \ref{monocover} contains the open balls $1+9\Z_3$, $8+9\Z_3$, $3^3+3^5\Z_3$ and $8\cdot 3^3+3^5\Z_3$. 

For the reduction at $2$ we look at the tables page 358--359 of \cite{kraus}. The equation is minimal for $-2\leq v_2(s)\leq 4$. We treat some of the possibilities. 
\begin{itemize}[leftmargin=*]
\item $v_2(s)=0$:~ $\Phi_{E_s,2}=\Z/3\Z$ if $s\equiv 1\mod 4$ and $\Phi_{E_s,2}=\Z/6\Z$ otherwise.
\item $v_2(s)=1$:~ $\Phi_{E_s,2}=\Z/2\Z$. 
\item $v_2(s)=2$:~ $\Phi_{E_s,2}=\Z/3\Z$ if $\frac{s}{2^2}\equiv -1\mod 4$ and $\Phi_{E_s,2}=\operatorname{SL}_2 (\F_3)$ otherwise.
\end{itemize}
As for $p=3$ these conditions give open balls contained in the open subsets of the covering given by Theorem \ref{monocover}. 

We can thus choose elliptic curves with prescribed finite monodromy at $2$ and $3$ in the list of possibilities. For example the curve
$$E_{4}\colon y^2=x^3+4$$
has maximal finite monodromy at $2$ and $3$, i.e. $\Phi_{E_s,2}=\operatorname{SL}_2 (\F_3)$ and $\Phi_{E_s,3}= \Z/3\Z \rtimes \Z/4\Z$. 
\end{ex}

\section{Existence of abelian varieties with maximal $d(A)$}

We will consider the universal abelian scheme $Z_g\rightarrow H_g$ of principally polarised and linearly rigidified abelian varieties of dimension $g$. The construction of the moduli space $H_g$ is done by D. Mumford in \cite{mumgit} chapter 7. Weak approximation for $H_g$ is not known so that we cannot use Corollary \ref{corweakapprox}. We bypass this difficulty by the following lemma based on a result of T. Ekedahl.
\begin{lem}\label{approxeke}
Let $\varphi\colon X\rightarrow Y$ be a finite étale cover of $K$-schemes of finite type for a number field $K$ with $X$ geometrically irreducible and $Y$ satisfying weak approximation. Let $v_1,\dots, v_n\in \Sigma_K$ and $U_i\subset X(K_{v_i})$ be a nonempty open subset for each $i\in \{1,\dots, n\}$. Then there is a finite extension $L/K$ with places $w_i\mid v_i$ for each $i\in \{1,\dots, n\}$ with $L_{w_i}=K_{v_i}$ and an $x\in X(L)$ such that $x_{L_{w_i}}\in U_i$ for each $i\in \{1,\dots, n\}$.   
\end{lem}
\begin{proof}
As $\varphi$ is a finite étale cover the sets $\varphi(U_i)$ are open in $Y(K_{v_i})$ for each $i\in \{1,\dots, n\}$. By assumption $Y$ satisfies weak approximation and $X$ is geometrically irreducible so that by Theorem 1.3 of \cite{eke} there is a point $y\in Y(K)$ with $y_{K_{v_i}} \in \varphi(U_i)$ for each $i\in \{1,\dots, n\}$ and $y$ has a connected fiber. As $\varphi$ is finite the point $x\in \varphi^{-1}(y)$ is an $L$-point of $X$ with $L/K$ a finite extension. 

Now for a fixed $i\in \{1,\dots, n\}$ the fiber of $y_{K_{v_i}}$ by $\varphi$ is given by $\Spec L\otimes K_{v_i}$. Furthermore, by construction $y_{K_{v_i}}$ is in $\varphi(U_i)$ so that there is a point $x_i\in U_i$ with $\varphi(x_i)=y_{K_{v_i}}$. Such a point corresponds to a place $w_i\mid v_i$ of $L$ with $L_{w_i}=K_{v_i}$ and so we have
$$x_{L_{w_i}}=x_i.$$
The point $x\in X(L)$ has the required properties.
\end{proof}

We can now prove our existence theorem for abelian varieties with maximal degree of semi-stability with respect to a given number field.

\begin{theo}
For any number field $K$ and nonzero integer $g$ there is a finite extension $M/K$ with a principally polarised abelian variety $A$ of dimension $g$ over $M$ such that
$$d(A)=d_g(K).$$
\end{theo}
\begin{proof}
Let $p_1, \dots, p_n$ be the prime divisors of $d_g(K)$. By Theorem \ref{monoglobdeg} there are principally polarised abelian varieties $B_1,\dots, B_n$ of dimension $g$ over $K$ and places $v_1,\dots, v_n$ with
$$v_{p_i} (\Card \Phi_{B_i,v_i})=v_{p_i} (d_g(K)).$$
Let $b_i\in H_g(K)$ be points with fiber $B_i$ for each $i\in \{1,\dots, n\}$. By Theorem \ref{monocover} applied to $Z_g\rightarrow H_g$ there is an open subset $U_i$ of $H_g(K_{v_i})$ with $(b_i)_{K_{v_i}}\in U_i$ and such that
$$s\in U_i \iff \Phi_{A_s,v_i}\simeq \Phi_{B_i,v_i}$$
where we denote by $A_s$ the abelian variety $(Z_g)_s$ given by a point $s\in H_g$. As $H_g$ is quasi-projective and geometrically irreducible there is a geometrically irreducible affine open subset $U\subset H_g$ with $b_1,\dots, b_n\in U$. 

Let $\varphi\colon U\rightarrow \A^k_K$ be the map given by Noether's normalization lemma. As $\varphi$ is generically étale there is a nonempty affine open subset $V\subset U$ and an open subset $Y\subset \A^k_K$ such that
$$\varphi_{|V}\colon V\longrightarrow Y$$
is finite étale. Let $F=U\setminus V$. The subsets $U_i\cap V(K_{v_i})=U_i\cap U(K_{v_i})\setminus F(K_{v_i})$ are open and nonempty. Indeed we have $b_i\in U_i\cap U(K_{v_i})$ by construction so that $U_i\cap U(K_{v_i})$ is nonempty open and $F$ is of positive codimension hence $F(K_{v_i})$ is of empty interior. Now Lemma \ref{approxeke} provides a point $s\in V(L)$ for a finite extension $L/K$ and places $w_i\in \Sigma_L$ with $L_{w_i}=K_{v_i}$ and $s_{L_{w_i}}\in U_i$ for each $i\in \{1,\dots n\}$.
The abelian variety $A_s$ given by $s\in H_g(L)$ verifies
$$\Phi_{A_s,w_i}=\Phi_{B_i,v_i}$$
by Theorem \ref{monocover}. It follows by Theorem \ref{monoglobdeg} that $d_g(K)\mid d(A_s)$.

To get equality, for any place $w\in \Sigma_L\setminus \{ w_1,\dots, w_n\}$ of bad reduction for $A_s$ consider the extension $M_w$ of $L_w$ of degree $\Card \Phi_{A_s,w}$ given by Lemma \ref{carfinimonogrp}. Let $d$ be the lowest common multiple of the degrees of the extensions $M_w$ obtained in this way and replace $M_w$ by its unramified extension of degree $d/\Card \Phi_{A,v}$ so that the extensions $M_w/L_w$ have all degree $d$ and are such that $(A_s)_{L_w}$ is semi-stable. Let $M_{w_i}/L_{w_i}$ be the unique unramified extension of degree $d$ for all $i\in \{1,\dots, n\}$ and $M$ be the field extension of $L$ given by Proposition \ref{weakapproxfield} applied to the local extensions $M_w$ for all $w\in \Sigma_L$ of bad reduction. We have
$$d((A_s)_{M})= \underset{w\in \Sigma_M}{\lcm} \Card \Phi_{A,w}$$
which by construction is $d_g(K)$ as $(A_s)_{M}$ has semi-stable reduction for all places of $\Sigma_M$ not above one of the $w_i$ and there is a unique place $W_i$ above each $w_i$ such that
$$\Phi_{(A_s)_{M},W_i}= \Phi_{B_i,v_i}.$$
\end{proof}

As stated in the introduction, from the proof of the theorem we get the following local-global result for finite monodromy groups.

\begin{theo}
Let $K$ be a number field and $g$ a nonzero integer. Let $G_1,\dots, G_n$ be finite groups such that there are places $v_1,\dots, v_n \in \Sigma_K$ and principally polarised abelian varieties $A_i$ of dimension $g$ over $K_{v_i}$ for all $1\leq i\leq n,$ such that
$$\Phi_{A_i,v_i}=G_i.$$
Then there is a finite extension $L$ of $K$ and places $w_1,\dots ,w_n\in \Sigma_L$ and a principally polarised abelian variety $A$ of dimension $g$ over $L$  such that for all $1\leq i\leq n$ we have
$$\Phi_{A, w_i}=G_i.$$
\end{theo}

We now prove Corollary \ref{maincor}.

\begin{proof}
Applying Theorem \ref{maintheo} to the field $K$ obtained by Theorem 1.1 of \cite{Ph2} gives an abelian variety $A$ over a number field $L$ with $\frac{M(2g)}{2^{g-1}}= d(A)$. 
The inequality $d_g\leq M(2g)$ follows from equality (\ref{lcm=d(A)}) and Corollary 6.3 of \cite{Sb1998}.
\end{proof}


\begin{thebibliography}{MMM}
\bibitem[Ek]{eke} T. Ekedahl. An effective version of Hilbert's irreducibility theorem. {\em Séminaire de théorie des nombres, Paris 1988-1989.} Progress in Mathematics. 91. Birkhäuser. Berlin. 1990. p. 241--248.


\bibitem[GIT]{mumgit} D. {Mumford}, J. {Fogarty} and F. {Kirwan}. {\em Geometric invariant theory. 3rd enl. ed.} Ergebnisse der Mathematik und ihrer Grenzgebiete. 34. Springer-Verlag. Berlin. 1993.

\bibitem[Ka]{Ka} Y. Katznelson. On the orders of finite subgroups of $\text{GL}(n,\mathbb{Z})$. {\em Expositiones Mathematicae}. 12. 1994. p. 453--457.

\bibitem[Kr]{kraus} A. Kraus. { Sur le défaut de semi-stabilité des courbes elliptiques à réduction additive.} {\em Manuscripta Mathematica.} 69. 1990. p. 353--385. 

	\bibitem[Ph]{Ph2} S.~Philip. Variétés abéliennes CM et grosse monodromie finie sauvage. \href{https://arxiv.org/abs/2103.04197}{arXiv:2103.04197}  (33p.)
	

\bibitem[Po]{po} B. Poonen. {\em Rational points on varieties.} Graduate Studies in Mathematics. AMS. Providence. 2017.

\bibitem[Ri]{rib} P. {Ribenboim}. {\em The theory of classical valuations.} Springer Monographs in Mathematics. Springer. New York. 1999.

\bibitem[SGA]{sga} Groupes de monodromie en géométrie algébrique. I. Séminaire de Géométrie Algébrique du Bois-Marie 1967--1969 (SGA 7 1). Dirigé par A. Grothendieck. Avec la collaboration de M. Raynaud et D. S. Rim. Lecture Notes in Mathematics. 288. Springer-Verlag. Berlin-New York. 1972.  

\bibitem[Se]{propgal} J.-P. Serre.  Propriétés galoisiennes des points d'ordre fini des courbes elliptiques. {\em Inventiones Mathematicae}. 15. 1972. p. 259--331. 


\bibitem[SZ]{Sb1998} A. Silverberg et Y. Zarhin.  Subgroups of inertia groups arising from abelian varieties. {\em Journal of Algebra}. 209. 1998. p. 403--420.

\bibitem[W]{Wils} J. S. Wilson. {\em Profinite Groups.} London Mathematical Society Monographs New Series, 19. Oxford University Press. Oxford. 1998.

\end{thebibliography}
\end{document}